\documentclass[12pt, a4, reqno]{amsart}
\usepackage{geometry}
\usepackage[all]{xy}
\usepackage{amsmath}
\usepackage{amscd}
\usepackage{amssymb}
\usepackage{amsthm}
\usepackage{comment}

\usepackage{mathrsfs}
\theoremstyle{definition}
\newtheorem{thm}{Theorem}[section]
\newtheorem{lem}[thm]{Lemma}

\newtheorem{prop}[thm]{Proposition}
\newtheorem{conj}[thm]{Conjecture}
\theoremstyle{definition}

\newtheorem{rem}[thm]{Remark}

\begin{document}

\title[GPC for Kummer surfaces of self-product CM type]
{Grothendieck's period conjecture for Kummer surfaces of self-product CM type}

\author{Daiki Kawabe}
\address{
Mathematical Institute, Tohoku University,
Aoba, Sendai, 980-8578, Japan
}

\curraddr{
Department of Mathematics, School of Education,
Waseda University,
Shinjuku, Tokyo, 169-8050, Japan
}
\email{kawabe@aoni.waseda.jp, daiki.kawabe.math@gmail.com}

\date{May 18, 2026}
\keywords{Grothendieck period conjecture, Kummer surfaces, CM elliptic curves}
\thanks{2010 {\text Mathematics Subject Classification.} 14C15, 14C30, 14C25}
\thanks{The author is supported by the JSPS KAKENHI Grant Number 18H03667 and 
Encouragement Program for Young Scientists 2022, Tohoku University}

\maketitle

\begin{abstract}
We show that the Grothendieck period conjecture holds for the Kummer surface associated with the square of a CM elliptic curve.
This means that the period isomorphism is dense in the torsor of motivic periods.
In other words, the isomorphism is dense in the torsor of motivated periods, and motivated classes on powers of the surface are algebraic.
The point is that the motive has a non-trivial transcendental part, but belongs to the Tannakian category generated by the motive of a CM elliptic curve.
\end{abstract}

\section{Introduction}
Let $\mathcal{V}_{k}$ be the category of smooth projective varieties
over a subfield $k \subseteq \mathbb{C}$.
Y.~Andr\'e \cite{AndreIHES} constructed the Tannakian category $\mathcal{M}_{k}^{\mathrm{And}}$ of pure motives with $\mathbb{Q}$-coefficients which contains the category $\mathcal{M}_{k}^{\mathrm{hom}}$ of homological motives over $k$
\footnote{The Tannakian category of mixed motives with $\mathbb{Q}$-coefficients
is well-defined in two different ways (but canonically equivalent) due to Ayoub \cite{ay1} and Nori \cite{N}, respectively.
For the details on the TC, see \cite{DM}.}.
If the standard conjecture of Lefschetz type holds, the categories $\mathcal{M}_{k}^{\mathrm{And}}$, $\mathcal{M}_{k}^{\mathrm{hom}}$ are equivalent. 
We call $\mathcal{M}_{k}^{\mathrm{And}}$ the category of Andr\'e motives.
Its objects are smooth projective varieties over $k$
and its morphisms are given by \textit{motivated} correspondences \cite[§4]{AndreIHES}.
A functor $h: \mathcal{V}_{k} \rightarrow \mathcal{M}_{k}^{\mathrm{And}}$ plays the role of universal cohomology.\\
\indent The categories $\mathcal{M}_{k}^{*}$ $(* \in \{\mathrm{And}, \mathrm{hom} \})$ admit 
two $\otimes$-functors (de Rham and Betti realizations) $H_{DR}, H_{B} : \mathcal{M}_{k}^{*} \rightarrow \mathrm{Vect}(\mathbb{Q})$.
There is an isomorphism of functors 
\[ \omega :  H_{DR} \otimes \mathbb{C} \overset{\cong}\rightarrow H_{B} \otimes \mathbb{C}. \]
More precisely, for any $M \in \mathcal{M}_{k}^{*}$, there is an isomorphism in $\mathrm{Vect}(\mathbb{C})$
\[ \omega_{M} :  H_{DR}(M) \otimes \mathbb{C} \overset{\cong}\rightarrow H_{B}(M) \otimes \mathbb{C}. \]
The entries of a matrix of $\omega_{M}$ w.r.t. some basis of $H_{DR}(M)$ (resp. $H_{B}(M)$) are called the \textit{periods} of $M$,
denoted by $P(M)$.
Let $\langle M \rangle$ be the Tannakian subcategory of $\mathcal{M}_{k}^{\mathrm{And}}$ generated by an Andr\'e motive $M \in \mathcal{M}_{k}^{\mathrm{And}}$. 
Its objects are given by algebraic constructions on $M$, i.e. sums, subquotients, duals, and tensor products.
We can define the \textit{motivated torsor of periods} of $M \in \mathcal{M}_{k}^{\mathrm{And}}$
 to be the scheme
\[ \Omega^{\mathrm{And}}_{M} : = {\mathrm{Isom}}^{\otimes}(H_{DR \mid \langle M \rangle},  H_{B \mid \langle M \rangle})  \]
of isomorphisms between the restrictions to $\langle M \rangle$ of 
the $\otimes$-functors $H_{DR}$ and $H_{B}$.
In particular, $\Omega^{\mathrm{And}}_{M}$ has a canonical $\mathbb{C}$-rational point 
$\omega_{M} \in \Omega^{\mathrm{And}}_{M}$.
Let $Z_{M}$ be the Zariski closure of $\omega_{M}$ in $\Omega^{\mathrm{And}}_{M}$.
The main topic of this paper is:
\begin{conj}\label{Motivic GC} (Grothendieck \cite[p.102, footnote]{Grothendieck})
\begin{enumerate}
\item Let $M$ be an Andr\'e motive over $\overline{\mathbb{Q}}$.
We say that $M$ satisfies the \textit{motivated version of the Grothendieck period conjecture} (MGPC for short)  if the inclusion  
\[ Z_{M} \subseteq \Omega^{\mathrm{And}}_{M} \]
 is an equality.
\item Let $X$ be a smooth projective variety over $\overline{\mathbb{Q}}$
\footnote{The GPC was mentioned in \cite[p.102, footnote]{Grothendieck} only for curves of any genus and abelian varieties, was formulated in \cite[Ch.~IX]{AndreG} for smooth projective varieties,
and was generalized in \cite[23.4.1]{AndreBook} for Andr\'e motives over a subfield of $\mathbb{C}$ of finite transcendental degree over $\mathbb{Q}$.
It also was formulated for Nori mixed motives and is equivalent to the Kontsevich-Zagier period conjecture \cite{KZ} in some sense 
(\cite{H-MS}, \cite[Prop.~5.15]{H}).
A $p$-adic analog of the GPC was formulated in \cite{Andrep}, \cite{AndreTate} 
by taking the ``sufficiently general'' embedding.
The functional analog of the GPC was proved by Ayoub (\cite{ay3}, \cite{ay4}) and Nori.
The historical background was written in the letter in \cite{Be20}.}.
We say that $X$ (or $h(X)$) satisfies the \textit{Grothendieck period conjecture} (GPC for short)  
if the inclusion 
\[ Z_{X} \subseteq \Omega^{\mathrm{And}}_{X} \]
is an equality and motivated classes on powers of $X$ are algebraic.
\end{enumerate}
\end{conj}
\begin{rem}
(1) is due to \cite[Prop.~7.5.2.2 \& 23.1.4.1]{AndreBook}. 
(2) is due to \cite[Ch.~IX]{AndreG} or \cite[Conj.~2.12]{BostCharles}.
(1) is equivalent to Conjecture \ref{MGPC} which is related to the degree of transcendence over $\mathbb{Q}$ of the field generated by the periods of $X$ (or $h(X)$).
In \cite[Def.~2.19 (4)]{BostCharles}, for $X \in \mathcal{V}_{\overline{\mathbb{Q}}}$, the torsor of \textit{motivic periods} $\Omega^{\mathrm{mot}}_{X}$ on $X$ is defined 
and satisfies $\Omega^{\mathrm{And}}_{X} \subseteq \Omega^{\mathrm{mot}}_{X}$.
Also, $X$ satisfies the GPC if and only if $Z_{X} \cong \Omega^{\mathrm{mot}}_{X}$.
In \cite{KSV}, the GPC was formulated for homological motives with coefficients.
\end{rem}
In more heuristic terms, the GPC says that
``polynomial relations with coefficients in $\overline{\mathbb{Q}}$ among the periods of a smooth projective variety $X$ over $\overline{\mathbb{Q}}$ are determined 
by the algebraic cycles on the powers of $X$''.\\
\indent There are many good works about the GPC, e.g. \cite{AndreGalois}, \cite{ay2}, \cite{Be02}, \cite{Be20}, \cite{BostCharles}, \cite{H-MS}, \cite{HW}, \cite{KZ}, and \cite{sv}. 
However, the GPC has so far only been fully established in the following cases:
\begin{enumerate}
\item $M = 1 := h(\mathrm{Spec}(\overline{\mathbb{Q}}))$: this is trivial.
\item $M = h(\mathbb{P}^{n})$: this follows from the transcendence of $\pi$.
\item $M = h(E)$ for a CM elliptic curve $E$: this is due to Chudnovsky \cite{Chudnovsky}.
\end{enumerate}

Here we explain the object of this paper.
Let $A$ be an abelian surface and $-id_{A}$ the inversion morphism.
The Kummer surface associated with $A$ is the minimal resolution of the quotient surface $A/\langle -id_{A} \rangle$.
As well known, it is a K3 surface.\\
\indent Our main theorem is:
\begin{thm} \label{main}
Let $A$ be an abelian surface over $\overline{\mathbb{Q}}$ isogenous to the square $E^{2}$ of a CM elliptic curve.
Let $\mathrm{Km}(A)$ be the Kummer surface over $\overline{\mathbb{Q}}$ associated with $A$.
Then the GPC holds for $\mathrm{Km}(A)$.
\end{thm}
This paper is organized as follows.
In§2, we give a formulation of the MGPC using the motivated Galois group.
We collect and prove some facts about the GPC, Mumford-Tate groups, and K3 surfaces.
In§3, we prove our main theorem.

\section{Motivated version of the Grothendieck period conjecture}
\subsection{Motivated Galois groups}
Let $k \subseteq \mathbb{C}$ be a subfield and $M \in \mathcal{M}_{k}^{\mathrm{And}}$ an Andr\'e motive.
We can define the \textit{motivated Galois group} of $M$ to be the group scheme
\[ G_{\mathrm{And}}(M) : = \mathrm{Aut}^{\otimes}(H_{B} \mid_{\langle M \rangle})  \]
of automorphisms of the restriction to $\langle M \rangle$ of the $\otimes$-functor $H_{B}$.
Note that $\Omega^{\mathrm{And}}_{M}$ is a torsor under $G_{\mathrm{And}}(M)$.
We write $\overline{\mathbb{Q}}(P(M))$ for the subfield of $\mathbb{C}$ genereted by the period $P(M)$ over $\overline{\mathbb{Q}}$.
Conjecture \ref{Motivic GC} (1) is equivalent to$:$ 
\begin{conj} (\cite[Prop.~7.5.2.2 \& 23.1.4.1]{AndreBook}) \label{MGPC} Let $M \in \mathcal{M}_{\overline{\mathbb{Q}}}^{\mathrm{And}}$ be an Andr\'e motive.
We say that $M$ satisfies the MGPC if $\Omega^{\mathrm{And}}_{M}$ is connected and 
\[ \mathrm{tr.deg}_{\mathbb{Q}} \overline{\mathbb{Q}}(P(M)) = \mathrm{dim} \ G_{\mathrm{And}}(M). \]
\end{conj}
\indent The irreducibility and connectedness of $\Omega^{\mathrm{And}}_{M}$ are equivalent since this is a smooth torsor.
The inequality $\leq$ is known.
\begin{prop} \label{inequality}
For any $M \in \mathcal{M}_{\overline{\mathbb{Q}}}^{\mathrm{And}}$, $\mathrm{tr.deg}_{\mathbb{Q}}\overline{\mathbb{Q}}(P(M)) \leq \mathrm{dim} \ G_{\mathrm{And}}(M)$.
\end{prop}
\begin{proof}
Note that $\mathrm{tr.deg}_{\mathbb{Q}}\overline{\mathbb{Q}}(P(M)) = \mathrm{dim}(Z_{M})$.
We have $\mathrm{dim} \ \Omega^{\mathrm{And}}_{M}  = \mathrm{dim} \ G_{\mathrm{And}}(M)$ 
since $\Omega^{\mathrm{And}}_{M}$ is a torsor under $G_{\mathrm{And}}(M)$.
So we get the inequality by $Z_{M} \subseteq \Omega^{\mathrm{And}}_{M}$.
\end{proof}
Thanks to the referees' suggestion, we provide a remark that will be used: 
\begin{rem} \label{trivial}
Note that if $M$ and $N$ are motives in $\mathcal{M}_{\overline{\mathbb{Q}}}^{\mathrm{And}}$
such that $\langle M \rangle \cong \langle N \rangle$, then $G_{\mathrm{And}}(M) \cong G_{\mathrm{And}}(N)$.
In particular, if additionally $\mathrm{tr.deg}_{\mathbb{Q}} \overline{\mathbb{Q}}(P(M)) = \mathrm{tr.deg}_{\mathbb{Q}}\overline{\mathbb{Q}}(P(N))$ and $G_{\mathrm{And}}(M)$ is connected, we see that the MGPC for $M$ is equivalent to the MGPC for $N$.
\end{rem}

Finally, we provide an elementary lemma.
\begin{lem} \label{times}
Let $M$ and $N$ be motives in $\mathcal{M}_{\overline{\mathbb{Q}}}^{\mathrm{And}}$.
\begin{enumerate}
\item $\overline{\mathbb{Q}}(P(M \oplus N)) = \overline{\mathbb{Q}}(P(M), P(N))$.
\item $\mathrm{tr.deg}_{\mathbb{Q}}\overline{\mathbb{Q}}(P(M^{\otimes n})) = 
\mathrm{tr.deg}_{\mathbb{Q}}\overline{\mathbb{Q}}(P(M))$ for $n \geq 1$.
\end{enumerate}
\end{lem}
\begin{proof}
(1) is clear.
(2) For our aim, we only prove the case where $n = 2$ and $\mathrm{dim}(H_{DR}^{*}(M)) = 2$.
Then we set $K : = \overline{\mathbb{Q}}(P(M^{\otimes 2}))$ and $L : = \overline{\mathbb{Q}}(P(M))$.
Taking bases of $H^{*}_{DR}(M)$ and $H^{*}_{B}(M)$, 
we set $\overline{\mathbb{Q}}(P(M)) = \overline{\mathbb{Q}}(p_{11}, p_{12}, p_{21}, p_{22})$.
Then 
\[
K = \overline{\mathbb{Q}}( \{ p_{ij}p_{kl} \ | \ 1 \leq i, j, k, l \leq 2 \} ) 
\ \ \ \subset \ \ \ \overline{\mathbb{Q}}(p_{11}, p_{12}, p_{21}, p_{22}) = L. \]
Since $p_{ij}^{2} \in K$ for any $i, j$, the elements $p_{ij} \in L$ are algebraic over $K$, and hence $L/K$ is algebraic. Thus we have
$\mathrm{tr.deg}_{\mathbb{Q}}(L) = \mathrm{tr. deg}_{\mathbb{Q}}(K)$.
\end{proof}

\subsection{Mumford-Tate groups}
We recall some properties of Mumford-Tate groups.\\
\indent A pure $\mathbb{Q}$-Hodge structure consists of a finite-dimensional $\mathbb{Q}$-vector space $H$
together with a decomposition $H_{\mathbb{C}} = \oplus_{i, j} H^{i, j}$ such that $\overline{H^{i, j}} = H^{j, i}$.
In other words, $H$ is a finite-dimensional $\mathbb{Q}$-vector space equipped with a representation 
$h : \mathbb{S} : = \mathrm{Res}_{\mathbb{C}/\mathbb{R}}\mathbb{G}_{m, \mathbb{C}} \rightarrow GL(H_{\mathbb{R}})$
such that via the composition $\mathbb{G}_{m, \mathbb{R}} \hookrightarrow \mathbb{S} \rightarrow  GL(H_{\mathbb{R}})$
an element $t \in \mathbb{G}_{m, \mathbb{R}}$ acts as $t^{-n} \cdot \mathrm{Id}$.
The \textit{Mumford-Tate group} $MT(H)$ of $H$
is the smallest algebraic subgroup of $GL(H)$ over $\mathbb{Q}$
such that $h(\mathbb{S}) \subset MT(H_{\mathbb{R}})$.
By definition, any MT group is always connected.
\begin{prop} \label{property of mt}
If $H_{1}$ and $H_{2}$ are $\mathbb{Q}$-Hodge structures, then $MT(H_{1} \oplus H_{2}) \subset MT(H_{1}) \times MT(H_{2})$
as subgroups of $GL(H_{1} \oplus H_{2})$, and the projection maps $p_{i} : MT(H_{1} \oplus H_{2}) \rightarrow MT(H_{i})$
are surjective.
\end{prop}
\begin{proof}
The assertions follow from the definition of the MT groups.
\end{proof}
Let $HS$ be the Tannakian category of pure Hodge structures 
with the forgetful fiber functor $\omega$ associating to $H$ the underlying $\mathbb{Q}$-vector space (see \cite{Deligne}).\\
\indent The Mumford-Tate group of $H$ can be described as
\[ MT(H) = \mathrm{Aut}^{\otimes}(\omega \mid_{\langle H \rangle}). \]
When $H$ is polarizable, $\langle H \rangle$ is semi-simple, so $MT(H)$ is reductive.\\
\indent Let $k \subseteq \mathbb{C}$ be a subfield.  Let $\mathcal{M}_{k, 1}^{\mathrm{And}}$ be the category of 1-motives, 
i.e. it is a thick abelian subcategory of $\mathcal{M}_{k}^{\mathrm{And}}$
whose object is a direct sum of the motive $h_{1}(A)$ of an abelian variety, 
the copies of the Lefschetz motive $\mathbb{L}$, and the unit motive $1$.
Let $\mathcal{M}_{k, 1}^{\mathrm{And}, \otimes}$ be the Tannakian subcategory of 
$\mathcal{M}_{k}^{\mathrm{And}}$ generated by 1-motives,
i.e. it is closed under sums, subquotients, duals, and tensor products.
This is equivalent to the Tannakian category generated by the motives of abelian varieties.
Recall that for smooth projective varieties
\begin{center}
 ``algebraic class $\Rightarrow$ motivated class $\Rightarrow$ absolute Hodge class $\Rightarrow$ Hodge class''
 \end{center}
The Hodge conjecture says that any Hodge class is algebraic. Andr\'e showed:
\begin{thm} (\cite{AndreIHES}) \label{mg = mt}
Let $k = \overline{k} \subset \mathbb{C}$.
The Hodge realization $\mathcal{M}_{k, 1}^{\mathrm{And}, \otimes} \rightarrow HS$ is fully faithful
\footnote{Andr\'e extended Thm.~\ref{mg = mt} to the Hodge realization from the category of mixed motives \cite{Andre1motives}.}.
In particular, the MG group of any motive in $\mathcal{M}_{k, 1}^{\mathrm{And}, \otimes}$ coincides with its MT group and any Hodge class on abelian varieties is motivated.
\end{thm}

\begin{lem} \label{sum}
Let $M$ and $N$ be motives in $\mathcal{M}_{\overline{\mathbb{Q}}, 1}^{\mathrm{And}, \otimes}$.
Assume $\mathrm{tr.deg}_{\mathbb{Q}}\overline{\mathbb{Q}}(P(M \oplus N)) 
= \mathrm{tr.deg}_{\mathbb{Q}}\overline{\mathbb{Q}}(P(M))$.
If the MGPC holds for $M \oplus N$, then it also holds for $M$.
\end{lem}
\begin{proof}
Since $M$, $N \in \mathcal{M}_{\overline{\mathbb{Q}}, 1}^{\mathrm{And}, \otimes}$, their MG groups and MT groups coincide by Theorem \ref{mg = mt}, so $\Omega^{\mathrm{And}}_{M}$ is connected.
We have $\mathrm{dim} \ MT(M) \leq \mathrm{dim} \ MT(M \oplus N)$ by Proposition \ref{property of mt}.
Also, we have $\mathrm{tr.deg}_{\mathbb{Q}}\overline{\mathbb{Q}}(P(M)) \leq \mathrm{dim} \ MT(M)$ and 
$\mathrm{tr.deg}_{\mathbb{Q}}\overline{\mathbb{Q}}(P(M \oplus N)) \leq \mathrm{dim} \ MT(M \oplus N)$ by Proposition \ref{inequality}.
Also, we have $\mathrm{tr.deg}_{\mathbb{Q}}\overline{\mathbb{Q}}(P(M \oplus N)) = \mathrm{tr.deg}_{\mathbb{Q}}\overline{\mathbb{Q}}(P(M))$ by assumption.
Then, as the MGPC holds for $M \oplus N$ by assumption, we conclude that it also holds for $M$.
\end{proof}

\subsection{The power of an elliptic curve} 
The following is due to Chudnovsky:
\begin{thm} (\cite{Chudnovsky}) \label{CM}
For a CM elliptic curve $E$ over $\overline{\mathbb{Q}}$, the GPC holds for $h(E)$.
\end{thm}
Let us immediately clarify to the reader that we can not prove the GPC for non-CM elliptic curves.
Instead, in this paper, we consider:
\begin{center}
``Does Chudnovsky's result imply the GPC for other cases?''
\end{center}
More precisely, for non-pairwise isogenous CM elliptic curves $E$, $E'$ and a positive integer $n$, we have natural questions:
\begin{enumerate}
\item Does the GPC hold for $h(E^{n})$?
\item Does the GPC hold for $M \in \langle h(E) \rangle$?
\item Does the GPC hold for $M \in \langle h(E), h(E') \rangle$?
\end{enumerate}
Unfortunately, we answer only (1) and (2).
Theorem \ref{main} is our answer to (2).
To answer (1), we prove that the GPC is stable under taking the power\footnote{This paper was submitted to J.~Reine~Angew.~Math. (Crelle), but was rejected. The editor, Prof.~Daniel Huybrechts passed the author on the comments of an anonymous referee of Crelle. The referee suggested and proved Prop.~\ref{power}.}:

\begin{prop} \label{power}
Let $n \geq 1$ be an integer, $X$ a smooth projective variety over $\overline{\mathbb{Q}}$, 
and $X^{n}$ its $n$-th power.
If the GPC holds for $h(X)$, then it also holds for $h(X^{n})$.
\end{prop}
\begin{proof}
By assumption, we see that the MGPC holds for $X$ and motivated classes on $X^{n}$ are algebraic. So it remains to prove that the MGPC holds for $X^{n}$.
Indeed, we have $\mathrm{tr.deg}_{\mathbb{Q}}\overline{\mathbb{Q}}(P(h(X)^{\otimes n})) = 
\mathrm{tr.deg}_{\mathbb{Q}}\overline{\mathbb{Q}}(P(h(X)))$ by Lemma \ref{times} (2).
Also, we have $\langle h(X^{n}) \rangle \cong \langle h(X) \rangle$ (using the fact that $h(X^{n}) \cong h(X)^{\otimes n}$
by definition of tensor product of motives, and that $h(X^{n})$ contains $h(X)$ as a direct summand.)
Then, as $G_{\mathrm{And}}(X)$ is connected by assumption, we see that the MGPC holds for $X^{n}$ by Remark \ref{trivial}. Thus the GPC holds for $X^{n}$.
\end{proof}

\begin{rem} 
Naturally, it is expected that the GPC is stable under other operations, e.g. direct summand and tensor product.
However, we could not prove it in general. 
\end{rem}

Now, we answer the question (1) which is already known (cf. \cite[Ch.~IX]{AndreG}) $:$
\begin{prop} \label{power of CM}
Let $n \geq 1$ be an integer and $E$ a CM elliptic curve over $\overline{\mathbb{Q}}$.
Then the GPC holds for $h(E^{n})$.
\end{prop}
\begin{proof}
It follows from Theorem \ref{CM} and Proposition \ref{power}.
\end{proof}

\begin{rem}
As far as we know,\footnote{Huber and W\"ustholz proved the KZ period conjecture for ``1-motives'' \cite{HW}.
Roughly speaking, their result ``corresponds'' to the GPC for linear relations of periods of 1-motives.} the GPC is unknown for the product of two non-isogenous CM elliptic curves.
For more details, see \cite[Rem.~3 (c)]{Kah}, \cite[Rem.~5.15]{KSV}.
The weak version of the GPC\footnote{This is called the ``de Rham-Betti conjecture'' and is weaker than the MGPC.
More precisely, for a smooth projective variety $X$ over $\overline{\mathbb{Q}}$, there is a chain of inclusions 
$Z_{X} \subseteq \Omega_{X} \subseteq \Omega_{X}^{\mathrm{dRB}} \subseteq \Omega^{\mathrm{And}}_{X} \subseteq \Omega^{\mathrm{mot}}_{X}$. We say that $X$ satisfies the de Rham-Betti conjecture if $\Omega_{X}^{\mathrm{dRB}} \cong \Omega^{\mathrm{And}}_{X}$.}
is proved for products of ``elliptic curves'' (\cite{sv}, \cite{Kah}\footnote{Kahn proved the fullness conjectures for products of elliptic curves.
His result contains the Hodge, Tate, and dRB conjectures for them.
His proof is ``more uniform'' than the previous proofs in known cases.}, \cite{KSV}\footnote{Kreutz-Shen-Vial proved 
$\Omega_{X} \cong \Omega^{\mathrm{mot}}_{X}$ for $X$ a product of non-CM elliptic curves (\cite[Thm.~5.13 (ii)]{KSV}) and also proved many results around the de Rham-Betti conjecture with coefficients.}).
\end{rem}

For the reader's convenience, we recall the following elementary facts:
\begin{prop} \label{elliptic}
Let $E$ be an elliptic curve over $\overline{\mathbb{Q}}$ and let $n \in \mathbb{Z}_{\geq 1}$.
\begin{enumerate}
\item $h(E^n) \cong \oplus_{i}\wedge^{i}(h_{1}(E)^{\oplus n})$.
\item $\langle h(E^n) \rangle \cong \langle h(E) \rangle$.
\item $\overline{\mathbb{Q}}(P(h(E^n))) = \overline{\mathbb{Q}}(P(h(E)))$.
\end{enumerate}
\end{prop}
\begin{proof}
$(1)$
Since $E^{n}$ is an abelian variety, $h(E^n) \cong \oplus_{i}\wedge^{i}(h_{1}(E^{n}))$.
By K\"unneth formula, $h_{1}(E^{n}) = h_{1}(E)^{\oplus n}$. Note that $\wedge^{2}h_{1}(E)  = \mathbb{L}$. Thus we get $(1)$.\\
\indent $(2)$ By (1),  $\langle h(E^n) \rangle \cong \langle \oplus_{i}\wedge^{i}(h_{1}(E)^{\oplus n}) \rangle 
\cong \langle 1 \oplus h_{1}(E) \oplus \mathbb{L} \rangle \cong \langle h(E) \rangle$.\\
\indent $(3)$ By (1), 
$\overline{\mathbb{Q}}(P(h(E^n))) = \overline{\mathbb{Q}}(P(\oplus_{i}\wedge^{i}(h_{1}(E)^{\oplus n}))
= \overline{\mathbb{Q}}(P(h_{1}(E) \oplus \mathbb{L})) = \overline{\mathbb{Q}}(P(h(E)))$.\\
\end{proof}

For a surface $S$, let $h(S) = \oplus_{i}^{4}h_{i}(S)$ denote the CK-decomposition as in \cite{Murre}.\\
Moreover, let $t_{2}(S)$ be the transcendental part of $h_{2}(S)$,
i.e. the orthogonal complement of $\mathbb{L}^{\oplus \rho(S)}$ in $h_{2}(S)$ as in \cite{KMP}.\\
\indent The following result plays a key role in the proof of Theorem \ref{main}$:$

\begin{prop} \label{t2}
Let $E$ be an elliptic curve $\overline{\mathbb{Q}}$.
\begin{enumerate}
\item $\mathrm{tr.deg}_{\mathbb{Q}}\overline{\mathbb{Q}}(P(t_{2}(E^{2}) \oplus \mathbb{L}))
= \mathrm{tr. deg}_{\mathbb{Q}}\overline{\mathbb{Q}}(P(h(E^{2})))$
\item If $E$ has CM, then the GPC holds for $t_{2}(E^{2}) \oplus \mathbb{L}$.
\end{enumerate}
\end{prop}
\begin{proof}
(1) We consider two expressions for $h_{2}(E^{2})$.
One is 
\[ h_{2}(E^{2}) \cong \mathbb{L}^{\oplus 2} \oplus h_{1}(E)^{\otimes 2} 
\cong \mathbb{L}^{\oplus 3} \oplus \mathrm{Sym}^{2}(h_{1}(E)) \]
by K\"unneth formula.
The other is 
\[ h_{2}(E^{2}) \cong \mathbb{L}^{\oplus \rho(E^{2})} \oplus t_{2}(E^{2}) \ \ (\rho(E^{2}) = 3, 4) \]
by definition of $t_{2}$.
The expressions induce
\[ \overline{\mathbb{Q}}(P(t_{2}(E^{2}) \oplus \mathbb{L})) \cong 
\overline{\mathbb{Q}}(P(h_{1}(E)^{\otimes 2} \oplus \mathbb{L})). \]
Then we have $\mathrm{tr.deg}_{\mathbb{Q}}\overline{\mathbb{Q}}(P(t_{2}(E^{2}) \oplus \mathbb{L}))
= \mathrm{tr. deg}_{\mathbb{Q}}\overline{\mathbb{Q}}(P(h_{1}(E) \oplus \mathbb{L}))$ by Lemma \ref{times} (2), so get
$\mathrm{tr.deg}_{\mathbb{Q}}\overline{\mathbb{Q}}(P(t_{2}(E^{2}) \oplus \mathbb{L}))
= \mathrm{tr.deg}_{\mathbb{Q}}\overline{\mathbb{Q}}(P(h(E^{2})))$ by Proposition \ref{elliptic} (3).\\
\indent (2) We have a decomposition
$h(E^{2}) = t_{2}(E^{2}) \oplus \mathbb{L} \oplus M$
with $M = \oplus_{i \neq 2}h_{i}(E^{2}) \oplus \mathbb{L}^{\oplus(\rho(E^{2}) - 1)}$ by the CK-decomposition. 
Since $E$ has CM, the GPC holds for $h(E^{2})$ by Proposition \ref{power of CM}.
Since $\mathrm{tr.deg}_{\mathbb{Q}}\overline{\mathbb{Q}}(P(t_{2}(E^{2}) \oplus \mathbb{L}))
= \mathrm{tr. deg}_{\mathbb{Q}}\overline{\mathbb{Q}}(P(h(E^{2})))$ by (1), 
the MGPC holds for $t_{2}(E^{2}) \oplus \mathbb{L}$ by Lemma \ref{sum}.
Thus, it remains to show that motivated classes on $E^{2}$ are algebraic.
Indeed, this follows from the fact that motivated classes are Hodge and that 
the Hodge conjecture holds for $E^{2}$ by \cite{tate} or \cite{murasaki}.
Thus the GPC holds for $t_{2}(E^{2}) \oplus \mathbb{L}$.
\end{proof}

\subsection{Surfaces with $b_{1} = b_{3} = 0$}

Throughout this subsection, let $k = \overline{\mathbb{Q}}$.
\begin{lem} \label{surface b1}
Let $S$ be a surface with $b_{1} = b_{3} = 0$.
If the GPC holds for $t_{2}(S) \oplus \mathbb{L}$, then it also holds for $h(S)$.
\end{lem}
\begin{proof}
By $b_{1} = b_{3} = 0$ and the CK-decomposition, we have
$h(S) \cong 1 \oplus  t_{2}(S) \oplus \mathbb{L}^{\oplus \rho(S)} \oplus \mathbb{L}^{\otimes 2}$ with $\rho > 0$.
Thus we have $\langle h(S) \rangle \cong \langle t_{2}(S) \oplus \mathbb{L} \rangle$ and 
$\overline{\mathbb{Q}}(P(h(S))) = \overline{\mathbb{Q}}(P(t_{2}(S) \oplus \mathbb{L}))$. 
By assumption, the MGPC holds for $t_{2}(S) \oplus \mathbb{L}$, $G_{\mathrm{And}}(t_{2}(S) \oplus \mathbb{L})$ is connected, and motivated classes on powers of $S$ are algebraic.
Then we conclude that the GPC holds for $S$ by Remark \ref{trivial}.
\end{proof}
For example, any K3 surface satisfies $b_{1} = b_{3} = 0$.
The standard reference for K3 surfaces is \cite{HuybrechtsK3book}.
For the reader's understanding, we provide non-trivial geometric examples that satisfy the GPC.
To do this, we recall a motivic formulation of Bloch's conjecture (BC for short):
\begin{conj} (\cite{Bloch}) \label{bloch}
Let $S$ be a surface. If $p_{g}(S) = 0$, then $t_{2}(S) = 0$.
\end{conj}
It holds for surfaces not of general type \cite{BKL}.
We mention the relation between the GPC and the BC for surfaces, which is not used for our aim.

\begin{prop}
Let $S$ be a surface with $b_{1} = b_{3} = 0$.
If $t_{2}(S) = 0$, then the GPC holds for $S$.
In particular, if the BC holds for a surface $S$ with $p_{g} = q = 0$ (e.g. an Enriques surface), then the GPC holds for $S$.
\end{prop}
\begin{proof}
It follows from Lemma \ref{surface b1}.
\end{proof}

Let us recall the object of this paper.
Let $A$ be an abelian surface.
The Kummer surface associated with $A$ 
is the K3 surface obtained by the minimal resolution of the quotient surface $A/\langle -id_{A} \rangle$.
This is an easy, but instructive example of K3 surfaces.
For more details, see \cite{o}, \cite{ss}.
We will use the following fact for our aim.

\begin{prop} (\cite{KMP}) \label{kummer}
Let $\mathrm{Km}(A)$ be the Kummer surface associated with an abelian surface $A$.
Then $t_{2}(\mathrm{Km}(A)) \cong t_{2}(A)$. In particular, $t_{2}(\mathrm{Km}(A)) \neq 0$.
\end{prop}

\section{Proof of the main theorem and remarks}
Here, we give an elementary proof for the main theorem of this paper.\\

\indent \textbf{Proof of Theorem \ref{main}.}
Set $S = \mathrm{Km}(A)$. Since $S$ is a K3 surface, by Lemma \ref{surface b1},
it suffices to show that the GPC holds for $t_{2}(S) \oplus \mathbb{L}$.
By Proposition \ref{kummer}, we may assume that $S = A$.
Since isogenous abelian varieties have isomorphic motives (\cite[Thm.~3.1]{DenM}),
we may further assume that $A = E^{2}$.
By Proposition \ref{t2} (2), the GPC holds for $t_{2}(E^{2}) \oplus \mathbb{L}$.
Thus, the GPC holds for $h(S)$. This ends the proof.

\begin{rem}
The motive $h(\mathrm{Km}(A))$ belongs to the Tannakian category 
$\mathcal{M}_{\overline{\mathbb{Q}}, 1}^{\mathrm{And}, \otimes}$
(more precisely $\langle h(E) \rangle$ for a CM elliptic curve $E$), 
but does not belong to the thick abelian subcategory $\mathcal{M}_{\overline{\mathbb{Q}}, 1}^{\mathrm{And}}$.
Unfortunately, we are aware that we can not provide any example of a motive 
that does not belong to $\langle h(E) \rangle$ and satisfies the GPC and $t_{2} \neq 0$.
\end{rem}
Kreutz-Shen-Vial proved the GPC for K3 surfaces of Picard corank 0 
($\rho = h^{1, 1} = 20$) \cite[Cor.~6.17 (i)]{KSV}.
Their result generalizes Theorem \ref{main}.
The details are as follows.
Let $S$ and $E$ be as in Theorem \ref{main}. Then $\rho(S) = h^{1, 1}(S) = 20$ since $\rho(E^{2}) = 4$.
Recall that a K3 surface $X$ over $\mathbb{C}$ admits a Shioda-Inose structure if 
there is a Nikulin involution $i$ on $X$ such that the desingularization of the quotient surface $X/\langle i \rangle$
is the Kummer surface $\mathrm{Km}(A')$ associated to an abelian surface $A'$.
Pedrini showed:
\begin{thm} \label{SIt} (\cite{p})
Let $X$ be a K3 surface over $\mathbb{C}$ which admits a Shioda-Inose structure.
Then $t_{2}(X) \cong t_{2}(\mathrm{Km}(A')) \cong t_{2}(A')$.
\end{thm}
\begin{proof}
This follows from the proof in \cite[Thm.~2]{p}.
\end{proof}
Let $X$ be a K3 surface of Picard rank $\rho(X) = 20$.
Then, by \cite[Cor.~6.4]{mor}, $X$ has a SI structure and, by \cite{Kat},
the abelian surface $A'$ associated to $X$ is isogenous to $E'^{2}$,
where $E'$ is a CM elliptic curve.
By \cite[Cor.~2.10 (i) \& 4.4 (i)]{mor}, $X$ is not necessarily isomorphic to the Kummer surface $\mathrm{Km}(A')$.
However, $h(X) \cong h(\mathrm{Km}(A'))$ by Theorem \ref{SIt}.
We could not prove the GPC for K3 surfaces with $\rho = 18, 19$.
\begin{rem}
Let $V$ be a hyper-K\"ahler variety over $\overline{\mathbb{Q}}$.
This means a variety over $\overline{\mathbb{Q}}$ whose base change to $\mathbb{C}$ is projective, 
irreducible holomorphic symplectic, and $b_{2}(V) > 3$. 
Kreutz-Shen-Vial proved that if $V$ has Picard corank 0 and is of known deformation type, then $V$ satisfies the MGPC \cite[Thm.~6.15 (i')]{KSV}.
We are interested in the GPC for HK varieties of dimension $\geq 3$,
but could not prove it.
\end{rem}

\textbf{Acknowledgement.}
I would like to thank Prof.~Nobuo Tsuzuki for suggesting the GPC and the book \cite{AndreG},
 and providing helpful comments during Friday seminars 2022. 
Also, I would like to thank Prof.~Charles Vial for his interest and fruitful discussions related to the GPC in Bielefeld in October 2022, and for pointing out mistakes in the first draft.
Finally, I would like to thank Prof.~Masaki Hanamura for his comments around 1-motives and the referees for their prompt responses and useful suggestions.

\end{document}